\documentclass[12pt,reqno]{amsart}

\usepackage[margin=1.5in]{geometry}  
\usepackage{graphicx}              
\usepackage{amsmath}               
\usepackage{amsfonts}              
\usepackage{amsthm}                

\newtheorem{thm}{Theorem}[section]
\newtheorem{lem}[thm]{Lemma}
\newtheorem{prop}[thm]{Proposition}
\newtheorem{rema}[thm]{Remark}

\newtheorem{conj}[thm]{Conjecture}

\newcommand{\R}{\mathbb{R}}

\begin{document}

\title[A one-dimensional symmetry result for nonlocal equations]
{A one-dimensional symmetry result for \\ a class of nonlocal semilinear equations \\ in the plane}

\thanks{The research leading
to these results has received funding from the
European Research Council
Grant n.~321186 - ReaDi -
``Reaction-Diffusion Equations, Propagation and Modelling''
and n.~277749 - EPSILON -
``Elliptic Pde's and Symmetry of
Interfaces and Layers for Odd Nonlinearities'',
the PRIN Grant n.~201274FYK7 ``Critical Point Theory
and Perturbative Methods for Nonlinear Differential Equations'',
the ANR ``HAB'' and ``NONLOCAL" projects (ANR-14-CE25-0013),
the Spanish Grant MTM2011-27739-C04-01 and the
Catalan Grant 2009SGR345. Part of this
work was carried out during a visit by F. Hamel to the
Weierstra{\ss} Institute, whose
hospitality is thankfully acknowledged.}

\author[F. Hamel]{Fran{\c{c}}ois Hamel}
\address{Aix Marseille Universit\'e,
CNRS, Centrale Marseille,
Institut de Math\'ematiques de Marseille,
UMR 7373, 13453 Marseille, France}
\email{francois.hamel@univ-amu.fr}

\author[X. Ros-Oton]{Xavier Ros-Oton}
\address{The University of Texas at Austin,
Department of Mathematics,
2515 Speedway, Austin TX 78751, USA}
\email{ros.oton@math.utexas.edu}

\author[Y. Sire]{Yannick Sire}
\address{Aix Marseille Universit\'e,
CNRS, Centrale Marseille,
Institut de Math\'ematiques de Marseille,
UMR 7373, 13453 Marseille, France}
\email{yannick.sire@univ-amu.fr}

\author[E. Valdinoci]{Enrico Valdinoci}
\address{Weierstra{\ss} Institute,
Mohrenstra{\ss}e 39, 10117 Berlin, Germany, and
Universit\`a di Milano,
Dipartimento di Matematica Federigo Enriques,
Via Cesare Saldini 50, 20133 Milano, Italia}
\email{enrico@mat.uniroma3.it}

\begin{abstract}
We consider entire solutions to ${\mathcal{L}} u= f(u)$ in $\R^2$, where $\mathcal L$ is a nonlocal operator with translation invariant,
even and compactly supported kernel $K$.
Under different assumptions on the operator $\mathcal L$, we show that monotone solutions are necessarily one-dimensional.
The proof is based on a Liouville type approach.
A variational characterization of the stability notion is also given, extending our results in some cases to stable solutions.
\end{abstract}

\maketitle




\bigskip


\section{Introduction}\label{intro}

In this paper, we consider solutions of an integral
equation driven by a nonlocal, linear operator of the form
\begin{equation}\label{OP}
{\mathcal{L}} u(x):= \int_{\R^n}\big( u(x)-u(y)\big)\,K(x-y)\,dy.
\end{equation}
We suppose
that~$K$ is a measurable and nonnegative kernel,
such that~$K(\zeta)=K(-\zeta)$ for a.e.~$\zeta\in\R^n$.
We consider both integrable and non-integrable kernels $K$.

We recall that
in the past few years, there has been an intense activity
in this type of operators, both for their mathematical interest
and for their applications in concrete models.
In particular, the fractional operators that we consider
here can be seen as a compactly supported version
of the fractional Laplacian~$(-\Delta)^s$
with~$s\in(0,1)$ (and possibly arising from a more
general kernel, which is not scale invariant and
does not possess equivalent extended problems).
Also, convolution operators are nowadays very
popular, also
in relation with biological models, see, among the
others~\cite{25, 26, 32, 34}.
\medskip

We consider here solutions~$u$ of the semilinear
equation
\begin{equation}\label{EQ}
{\mathcal{L}} u = f(u)\quad\hbox{ in }\R^2.
\end{equation}
Notice that, in the biological framework, the solution~$u$
of this equation
is often thought as the density of a biological species
and the nonlinearity~$f$ is a logistic map,
which prescribes the birth and death rate of the population.
In this setting,
the nonlocal diffusion modeled by~${\mathcal{L}}$
is motivated by the long-range interactions between
the individuals of the species.
\medskip

The goal of this paper is to study the symmetry properties
of solutions of~\eqref{EQ} in the light of a famous conjecture
of De Giorgi arising in elliptic partial differential
equations, see~\cite{EDG}. The original problem
consisted in the following question:

\begin{conj}\label{C:DG}
Let~$u$ be a bounded solution of
$$ -\Delta u=u-u^3$$
in the whole of~$\R^n$, with
\begin{equation*}
{\mbox{$\partial_{x_n} u(x)>0$
for any $x\in\R^n$.}}\end{equation*}
Then, $u$ is necessarily one-dimensional, i.e. there exist~$u_\star:\R\to\R$
and~$\omega\in \R^n$ such that~$u(x)=u_\star(\omega\cdot x)$,
for any~$x\in\R^n$, at least when~$n\le8$.
\end{conj}

The literature has presented several variations of Conjecture~\ref{C:DG}:
in particular, a weak form of it has been investigated
when the additional assumption
\begin{equation}\label{LIM}
\lim_{x_n\to\pm\infty} u(x_1,\dots,x_n)=\pm1
\end{equation}
is added to the hypotheses.
When the limit in~\eqref{LIM} is uniform with respect to the
variables~$(x_1,\dots,x_{n-1})\in\R^{n-1}$,
the version of Conjecture~\ref{C:DG} obtained in this
way is due to Gibbons and is related to problems in
cosmology.

In spite of the intense activity of the problem,
Conjecture~\ref{C:DG} is still open in its generality.
Up to now, Conjecture~\ref{C:DG} is known to have a positive
answer in dimension~$2$ and~$3$ (see~\cite{AC,GG} and
also~\cite{AAC,BCN})
and a negative answer in dimension~$9$ and higher (see~\cite{DKW}).
Also, the weak form of
Conjecture~\ref{C:DG} under the limit assumption in~\eqref{LIM}
was proved, up to the optimal dimension~$8$, in~\cite{S}
(see also~\cite{FVT} for more general conditions at infinity),
and the version of Conjecture~\ref{C:DG}
under a uniform limit assumption in~\eqref{LIM}
holds true in any dimension
(see~\cite{BBG, BHM,F}).
Since it is almost impossible to keep track in this short introduction
of all the research developed on this important topic,
we refer to~\cite{FV} for further details and motivations.
\medskip

The goal of this paper is to investigate
whether results in the spirit of
Conjecture~\ref{C:DG} hold true when the Laplace operator
is replaced by the nonlocal operator in~\eqref{OP}.
We remark that symmetry results in nonlocal settings
have been obtained in~\cite{CC1,CC2,CS2,CS3,CSM,LV,SV},
but all these works dealt with fractional
operators with scaling
properties at the origin and at infinity
(and somehow with
nice regularizing effects).

Also, some of the problems
considered in the previous works rely on an extension
property of the operator that brings the problem
into a local (though higher dimensional and either singular
or degenerate) problem (see however~\cite{bucur,CSV}
where symmetry results for fractional problems
have been obtained without extension techniques).

In this sense,
as far as we know, this paper is the first one to take into
account kernels that are compactly supported,
for which the above
regularization techniques do not always hold and for which
equivalent local problems are not available.
Moreover, the strategy used in our proof is different
from the ones already exploited in the nonlocal setting,
since it relies directly on a technique
introduced by~\cite{BCN} and refined in~\cite{AC},
which reduced the symmetry property of the level
sets of a solution to a Liouville type property
for an associated equation (of course, differently from
the classical case, we will have to deal with equations,
and in fact inequalities, of integral type, in which the
appropriate simplifications are more involved).
\medskip

In this paper, we prove the following
one-dimensional result in dimension~$2$. The case of dimension $3$, following the approach of Ambrosio and Cabr\'e in the local case for instance would require deeper analysis of optimal energy estimates.   
Here, and throughout the paper, $B_r$ denotes the open Euclidean ball with radius $r>0$ and centered at the origin, $B_r(x)=x+B_r$, and $\chi_E$ denotes the characteristic function of a set $E$.

\begin{thm}\label{DG-2}
Let $n=2$ and let $\mathcal L$ be an operator of the form \eqref{OP}, with $K$
satisfying either
\begin{equation}\label{ASS-kernel}
m_0\chi_{B_{r_0}}(\zeta)\le K(\zeta)\le
M_0\chi_{B_{R_0}}(\zeta)
\end{equation}
or
\begin{equation}\label{ASS-kernel-frac}
m_0\chi_{B_{r_0}}(\zeta)\le |\zeta|^{2+2s}\,K(\zeta)\le
M_0\chi_{B_{R_0}}(\zeta),
\end{equation}
for any~$\zeta\in\R^2$, for some fixed~$M_0\geq m_0>0$,~$R_0\geq r_0>0$, and $0<s<1$ in~\eqref{ASS-kernel-frac}.
Let~$u$ be a solution of~\eqref{EQ}, with~$u\in C^1(\R^2)$ and~$f\in C^{1,\alpha}(\R)$.
Assume that
\begin{equation}\label{mono}
{\mbox{$\partial_{x_2} u(x)>0$ for any $x\in\R^2$.}}
\end{equation}
Then, $u$ is necessarily one-dimensional.
\end{thm}

The assumptions in~\eqref{ASS-kernel}
and~\eqref{ASS-kernel-frac} correspond, respectively,
to the case of an integrable kernel of convolution
type and to the case of a cutoff fractional kernel.
For the existence and further properties of one-dimensional solutions of~\eqref{EQ}
under quite general conditions, see Theorem~3.1(b)
in~\cite{Bates}, and~\cite{ch,co}. As far as assumption~\eqref{ASS-kernel-frac} is concerned, there is no direct reference on the existence of one-dimensional solutions. However, an adaptation of the techniques in \cite{Palatucci-Savin-Valdinoci} could lead to such a result.

We recall that if condition~\eqref{ASS-kernel-frac}
(or, more generally,~(H1) below) is assumed,
one needs to interpret~\eqref{OP}
in the principal value sense, i.e., as customary,
\begin{eqnarray*}
{\mathcal{L}} u(x)&:=& {\rm P.V.}\;\int_{\R^n}
\big( u(x)-u(y)\big)\,K(x-y)\,dy\\
&:=& \lim_{r\to 0} \int_{\R^n\setminus B_r(x)}
\big( u(x)-u(y)\big)\,K(x-y)\,dy.
\end{eqnarray*}

As a matter of fact, our proof of Theorem \ref{DG-2} does not use any special structure of the kernel $K$, but only relies on the following facts: the kernel $K$ has compact support, and the operator $\mathcal L$ satisfies a Harnack inequality.
More precisely, we need:
\begin{itemize}
\item[(H1)] The operator $\mathcal L$ is of the form \eqref{OP}, with the kernel $K$
satisfying $K \geq 0$, $K(\zeta)=K(-\zeta)$ and $K(\zeta)\ge m_0\chi_{B_{r_0}}(\zeta)$
in $\R^2$ for some $m_0>0$ and $r_0>0$.
Moreover, $K$ has \emph{compact support} in $B_{R_0}$ for some $R_0>0$, that is,
\[K\equiv0\quad\textrm{in}\ \R^2\setminus B_{R_0},\]
and
\[\int_{B_{R_0}}|\zeta|^2 K(\zeta)d\zeta <\infty.\]

\item[(H2)] The operator $\mathcal L$ satisfies the following \emph{Harnack inequality}: if $\varphi$ is continuous and positive in $\R^2$ and is a weak solution to ${\mathcal{L}}\varphi+c(x)\varphi=0$ in $B_R$, with $c(x)\in L^\infty(B_1)$ and $\|c\|_{L^{\infty}(B_R)}\le b$, then
    \[\sup_{B_{R/2}}\varphi \leq C \inf_{B_{R/2}}\varphi\]
    for some constant $C$ depending on $\mathcal{L}$ and $b$, but independent of $\varphi$.
\end{itemize}


Under these assumptions, we have the following.

\begin{thm}\label{DG-3}
Let $n=2$, let $\mathcal L$ be an operator of the form \eqref{OP}, with $K$ and $\mathcal{L}$ satisfying~{\rm{(H1)}} and~{\rm{(H2)}}, and let~$u$ be a solution of~\eqref{EQ}, with~$u\in C^1(\R^2)$ and~$f\in C^1(\R)$.
Assume that
\begin{equation*}
{\mbox{$\partial_{x_2} u(x)>0$ for any $x\in\R^2$.}}
\end{equation*}
If $K$ is not integrable, assume in addition that $u\in C^3(\R^2)$.
Then, $u$ is necessarily one-dimensional.
\end{thm}

When \eqref{ASS-kernel} holds, then (H2) follows from the results of Coville (more precisely, Corollary~1.7 in~\cite{Coville}).
Similarly, when \eqref{ASS-kernel-frac} is in force, then~(H2)
follows from a suitable generalization of the results in~\cite{DKP}
(see Remark~\ref{TBD} below).
Thus, thanks to the results in \cite{Coville,DKP}, Theorem \ref{DG-2} follows from Theorem~\ref{DG-3} ---the only difference being the regularity assumed on the solution $u$.

Notice that when the kernel $K$ is non-integrable at the origin, then one expects the operator $\mathcal L$ to be regularizing, and thus bounded solutions $u$ to \eqref{EQ} to be at least $C^1$ (recall that $f$ is $C^1(\R)$). Moreover, when $f$ is smooth, then $u$ is expected to be smooth.
However, in case that $K$ is integrable at the origin as in (\ref{ASS-kernel}), then it is not clear if all bounded solutions are in $C^1(\R^2)$, and this is why we need to take this assumption in Theorem~\ref{DG-2}.

\begin{rema}
Notice that one can produce a $C^1$ solution by the following argument: rewrite equation \eqref{EQ} into the following form:
$$
\int_{\mathbb R^n} u(y)K(x-y)dy = u(x) - f(u(x)).
$$
Hence if $K$ is $C^1$, then the left hand-side of the equation is also $C^1$.
Therefore, assuming that the map $r \to r-f(r)$ is invertible with a $C^1$ inverse, leads to a $C^1$ solution $u$.
\end{rema}

\begin{rema}\label{TBD}
Thanks to the results of \cite{DKP}, the Harnack inequality holds for fractional truncated kernels as in~\eqref{ASS-kernel-frac} ---see $(2.2)$-$(2.3)$ in \cite{DKP}.
Moreover, a straightforward adaptation of their proof allows to take into account the $($bounded$)$ zero order term $c(x)$, and thus condition~{\rm{(H2)}} is satisfied for kernels $K$ satisfying \eqref{ASS-kernel-frac}.
\end{rema}

Harnack inequalities for general nonlocal operators $\mathcal L$ have been widely studied and are known for different classes of kernels $K$; see for instance a rather general form of the Harnack inequality in \cite{DKP}.
Notice that in our case, we need a Harnack inequality with a zero order term in the equation.
It has been proved when the integral operator is the pure fractional Laplacian in \cite{CS1} and refined in \cite{TX}.
It is by now well known that the Harnack inequality may fail depending on the kernel $K$ under consideration, and a characterization of the classes of kernels for which it holds is out of the scope of this paper.
Notice indeed that condition~\eqref{ASS-kernel} is stronger than~(H1),
but under the general assumption~(H1) 
then the Harnack inequality in~(H2) is not known, and thus needs
to be assumed
in Theorem~\ref{DG-3}.

The rest of the paper is devoted to the proof of Theorems~\ref{DG-2} and~\ref{DG-3}.
In particular, Section~\ref{PF}
will present the proof these results,
making use of suitable algebraic identities
and a Liouville type result in a nonlocal setting.
Then, in Section~\ref{EXT} we will
consider the extension of Theorem \ref{DG-3}
to stable (instead of monotone) solutions,
giving also a variational characterization of
stability.\footnote{This paper is the
outcome of two parallel and independent
projects developed at the same time
for these two classes of operators, see~\cite{HV,ROS}.
Since the motivation and the techniques
used are similar, we thought that it was
simpler to merge the two projects into
a single, and comprehensive, paper.}


\section{Proof of Theorems \ref{DG-2} and \ref{DG-3}}\label{PF}

The proofs of Theorems~\ref{DG-2} and~\ref{DG-3} are exactly the same.
We will prove them at the same time.
The first step towards the proof of these results is a suitable algebraic computation, that we
express in this result:

\begin{lem}\label{L-A-1}
Let~$u$ be as in Theorem~$\ref{DG-2}$ or~$\ref{DG-3}$.
Let~$u_i:=\partial_{x_i} u$, for~$i\in\{1,2\}$, and
\begin{equation}\label{DF:v}
v(x):=\frac{u_1(x)}{u_2(x)}.
\end{equation}
Also, let~$\tau\in C^\infty_0(\R^2)$.
Then
\begin{equation}\label{ALG-1}
\begin{split}
&\int_{\R^2}\int_{\R^2} \big(v(x)-v(y)\big)^2\,\tau^2(x) \,u_2(x)\,
u_2(y)\,K(x-y)\,dx\,dy \\ &\qquad= -
\int_{\R^2}\int_{\R^2} \big(v(x)-v(y)\big)\,
\big(\tau^2 (x)-
\tau^2 (y)\big)\,v(y)
\,u_2(x)\,u_2(y)\,K(x-y)\,dx\,dy.
\end{split}
\end{equation}
\end{lem}

\begin{proof}
First, notice that in case \eqref{ASS-kernel-frac}, since $f\in C^{1,\alpha}$ then $u\in C^{1+2s+\alpha}(\R^2)$.
This means that in all cases ---either \eqref{ASS-kernel} or \eqref{ASS-kernel-frac} or (H1)---, the derivatives $u_i$ are regular enough so that ${\mathcal{L}}u_i$ is well defined pointwise, and hence all the following integrals converge.

We observe that, for any~$g$ and $h$ regular enough,
\begin{equation}\label{R1}
\begin{array}{rcl}
\displaystyle2\int_{\R^2} {\mathcal{L}} h(x)\,g(x)\,dx
& = & \displaystyle2\int_{\R^2}\left[\int_{\R^2}
\big( h(x)-h(y)\big)\,K(x-y)\,dy\right]\,g(x)\,dx
\vspace{3pt}\\
& = &
\displaystyle\int_{\R^2}\left[\int_{\R^2}
\big( h(x)-h(y)\big)\,K(x-y)\,dy\right]\,g(x)\,dx
\vspace{3pt}\\
& & \displaystyle+
\int_{\R^2}\left[\int_{\R^2}
\big( h(y)-h(x)\big)\,K(x-y)\,dx\right]\,g(y)\,dy
\vspace{3pt}\\
& = & \displaystyle
\int_{\R^2}\int_{\R^2}
\big( h(x)-h(y)\big)\,
\big( g(x)-g(y)\big)\,
K(x-y)\,dx\,dy.
\end{array}
\end{equation}
By~\eqref{EQ}, we have that
\begin{equation}\label{5bis}\begin{array}{rcl}
f'\big(u(x)\big)\,u_i(x) & \!\!=\!\! &
\partial_{x_i} \left( f\big(u(x)\big)\right)\\
&  \!\!=\!\! & \displaystyle
\partial_{x_i} \big({\mathcal{L}} u(x)\big)
=\partial_{x_i} \left( \int_{\R^2}
\big( u(x)-u(x-\zeta)\big)\,K(\zeta)\,d\zeta \right)
\vspace{3pt}\\
&  \!\!=\!\! & \displaystyle\int_{\R^2}\!\!
\big( u_i(x)-u_i(x-\zeta)\big)\,K(\zeta)\,d\zeta
\vspace{3pt}\\
&  \!\!=\!\! & {\mathcal{L}} u_i(x).
\end{array}\end{equation}
Accordingly,
\begin{eqnarray*}
&& f'(u)\,u_1 u_2 = \big({\mathcal{L}} u_1\big) \,u_2\\
{\mbox{and }} && f'(u)\,u_1 u_2 = \big({\mathcal{L}} u_2\big) \,u_1.
\end{eqnarray*}
By subtracting these two identities
and using~\eqref{DF:v}, we obtain
$$ 0=\big({\mathcal{L}} u_1\big) \,u_2-
\big({\mathcal{L}} u_2\big) \,u_1 = \big({\mathcal{L}} (v u_2)\big) \,u_2
-\big({\mathcal{L}} u_2\big) \,(v u_2).$$
Now, we multiply the previous equality by~$2\tau^2 v$ and we
integrate over $\R^2$.
Recalling~\eqref{R1} together with $vu_2$, we conclude that
\begin{equation*}
\begin{split}
0\, &= 2\int_{\R^2} {\mathcal{L}} (v u_2)(x) \,\tau^2(x) v(x) u_2(x)\,dx
- 2\int_{\R^2}
{\mathcal{L}} u_2(x) \,\tau^2(x) v^2(x) u_2(x)\,dx\\
&=
\int_{\R^2}\!\int_{\R^2}\!\!
\big( v(x)u_2(x)\!-\!v(y)u_2(y)\big)
\big( \tau^2(x) v(x) u_2(x)\!-\!\tau^2(y) v(y) u_2(y)\big)\,
K(x\!-\!y)\,dx\,dy
\\ &\qquad-
\int_{\R^2}\!\int_{\R^2}\!\!
\big( u_2(x)-u_2(y)\big)\,
\big(\tau^2(x) v^2(x) u_2(x)-\tau^2(y) v^2(y) u_2(y)\big)
K(x-y)\,dx\,dy\\
&=: I_1-I_2.
\end{split}
\end{equation*}
By writing
$$ v(x)u_2(x)-v(y)u_2(y) =
\big(u_2(x)-u_2(y)\big)\,v(x) + \big(v(x)-v(y)\big)\, u_2(y),$$
we see that
\begin{equation}\label{XR-2}\begin{split}
I_1\,& =
\int_{\R^2}\!\int_{\R^2}\!\!
\big(u_2(x)\!-\!u_2(y)\big)\,\big( \tau^2(x) v(x) u_2(x)\!-\!\tau^2(y) v(y) u_2(y)\big)\,
v(x)\,K(x\!-\!y)\,dx\,dy
\\&\quad+
\int_{\R^2}\!\int_{\R^2}\!\! \big(v(x)\!-\!v(y)\big)\,
\big( \tau^2(x) v(x) u_2(x)\!-\!\tau^2(y) v(y) u_2(y)\big)\,
u_2(y)\,K(x\!-\!y)\,dx\,dy
.\end{split}\end{equation}
In the same way, if we write
$$\begin{array}{rcl}
\tau^2(x) v^2(x) u_2(x)\!-\!\tau^2(y) v^2(y) u_2(y) & \!\!\!=\!\!\! & \big(\tau^2(x) v(x) u_2(x)\!-\!\tau^2(y) v(y) u_2(y)\big)\,v(x)\vspace{3pt}\\
& & +\big(v(x)-v(y)\big)\,\tau^2(y) v(y) u_2(y),\end{array}$$
we get that
\begin{equation}\label{XR-3}\begin{split}
I_2\,&=\int_{\R^2}\int_{\R^2}\!\!
\big( u_2(x)\!-\!u_2(y)\big)\,
\big(
\tau^2(x) v(x) u_2(x)\!-\!
\tau^2(y) v(y) u_2(y)\big)\,v(x)
\,K(x\!-\!y)\,dx\,dy\\&\qquad+
\int_{\R^2}\int_{\R^2}
\big( u_2(x)\!-\!u_2(y)\big)\,
\big(v(x)\!-\!v(y)\big)\,\tau^2(y) v(y) u_2(y)\,K(x\!-\!y)\,dx\,dy
.\end{split}\end{equation}
By~\eqref{XR-2} and~\eqref{XR-3}, after a simplification we obtain that
\begin{eqnarray*}
I_1-I_2& \!\!\!=\!\!\! &
\int_{\R^2}\int_{\R^2}\!\! \big(v(x)\!-\!v(y)\big)\,
\big( \tau^2(x) v(x) u_2(x)\!-\!\tau^2(y) v(y) u_2(y)\big)\,
u_2(y)\,K(x\!-\!y)\,dx\,dy
\\ &&- \int_{\R^2}\int_{\R^2}\!\!
\big( u_2(x)\!-\!u_2(y)\big)\,
\big(v(x)\!-\!v(y)\big)\,\tau^2(y) v(y) u_2(y)\,K(x\!-\!y)\,dx\,dy.\end{eqnarray*}
Now we notice that
\[\begin{split}
\tau^2(x) v(x) u_2(x)&-\tau^2(y) v(y) u_2(y)= \big(v(x)-v(y)\big)\,\tau^2(x) \,u_2(x)+\\
&+\big(\tau^2 (x)- \tau^2(y) \big)\,u_2(x)\,v(y) +\big( u_2(x)-  u_2(y)\big)\,\tau^2(y)\,v(y),
\end{split}\]
and so
\begin{eqnarray*}
I_1-I_2&=&
\int_{\R^2}\int_{\R^2} \big(v(x)-v(y)\big)^2\,\tau^2(x) \,u_2(x)\,
u_2(y)\,K(x-y)\,dx\,dy \\ &&\hspace{-15mm}+
\int_{\R^2}\int_{\R^2} \big(v(x)-v(y)\big)\,
\big(\tau^2 (x)-
\tau^2 (y)\big)\,v(y)
\,u_2(x)\,u_2(y)\,K(x-y)\,dx\,dy
.\end{eqnarray*}
This proves~\eqref{ALG-1}.\end{proof}

Now we use a Liouville type approach to prove
that solutions~$v$ of
the integral equation in~\eqref{ALG-1} are necessarily
constant (and this is the only step
in which the assumption that the ambient space is~$\R^2$
plays a crucial role):

\begin{lem}\label{Lio}
Let~$u$ be as in Theorem~$\ref{DG-2}$ or~$\ref{DG-3}$, and let $v=u_1/u_2$.
Then~$v$ is constant.
\end{lem}

\begin{proof} First, by the previous Lemma $v$ satisfies~\eqref{ALG-1} for all $\tau\in C^\infty_c(\R^2)$.

Let~$R>1$, to be taken arbitrarily large
in the sequel.
Let~$\tau:=\tau_R\in C^\infty_0(B_{2R})$,
such that~$0\le\tau\le1$ in~$\R^2$,~$\tau=1$ in~$B_R$ and
\begin{equation}\label{tau}
|\nabla \tau|\le CR^{-1},
\end{equation}
for some~$C>0$ independent of~$R>1$.
Throughout the proof, $C$ will denote a positive constant which may change from a line to another, but which is independent of $R>1$.
Using~\eqref{ALG-1},
and recalling~\eqref{ASS-kernel},~\eqref{mono} and
the support properties
of~$\tau$, we observe that
\begin{equation}\label{J24}
\begin{split}
0&\le J_1\,\!:=\!\int_{\R^2}\int_{\R^2} \big(v(x)-v(y)\big)^2\,\tau^2(x) \,u_2(x)\,
u_2(y)\,K(x-y)\,dx\,dy \\ &\!\le\!
\iint_{{\mathcal{R}}_R}\!\!\big|v(x)\!-\!v(y)\big|\,
\big|\tau(x)\!-\!\tau (y)\big|\,
\big|\tau(x)\!+\!\tau (y)\big|\,
|v(y)|\,u_2(x)\,u_2(y)\,K(x\!-\!y)\,dx\,dy\\ &=:J_2
,\end{split}
\end{equation}
where
\begin{eqnarray*}
{{\mathcal{R}}_R} &:=& \{(x,y)\in\R^2\times\R^2 {\mbox{ s.t. }}
|x-y|\le R_0 \}\cap {{\mathcal{S}}_R}\qquad  {\mbox{and }}
\\
{\mathcal{S}}_R &:=& \Big( (B_{2R}\times B_{2R})\setminus (B_{R}\times B_{R}) \Big)
\;\cup \;\Big(
B_{2R}\times (\R^2\setminus B_{2R})\Big)\\
& & \; \cup\; \Big(
(\R^2\setminus B_{2R})\times B_{2R}\Big).
\end{eqnarray*}

Moreover, making use of the Cauchy-Schwarz inequality, we see that
\begin{equation}\label{J22}
\begin{split}
& J_2^2 \le
\iint_{{\mathcal{R}}_R} \big(v(x)-v(y)\big)^2\,
\big(\tau(x)+\tau (y)\big)^2\,
\,u_2(x)\,u_2(y)\,K(x-y)\,dx\,dy
\\&\qquad\qquad\cdot
\iint_{{\mathcal{R}}_R} \big(\tau(x)-\tau (y)\big)^2\,
v^2(y)\,u_2(x)\,u_2(y)\,K(x-y)\,dx\,dy
.\end{split}\end{equation}
Now we notice that
\begin{equation}\label{HA}
u_2(x)\le C\,u_2(y)
\end{equation}
for any~$(x,y)\in{{\mathcal{R}}_R}$, for a suitable~$C>0$, possibly depending on~$R_0$ but independent of~$R>1$ and~$(x,y)\in{\mathcal{R}}_R$. This is a consequence of~\eqref{5bis} with $f'(u)\in L^{\infty}(\R^2)$ and of assumption~(H2) applied recursively to some shifts of the continuous and positive function~$u_2$.

{F}rom~\eqref{tau},~\eqref{HA} and the assumption $v\,u_2\in L^{\infty}(\R^2)$,
we obtain that, for any $(x,y)\in{\mathcal{R}}_R$,
$$ \big(\tau(x)-\tau (y)\big)^2\,
v^2(y)\,u_2(x)\,u_2(y) \le CR^{-2} \,|x-y|^2\,v^2(y)\,u_2^2(y)\le CR^{-2}\,|x-y|^2,$$
for some~$C>0$ independent of $R>1$ (the constant $C$ in the last term may be larger than the one in the second term). Hence, by~\eqref{ASS-kernel},~(H1) and the symmetry in the $(x,y)$ variables,
$$\begin{array}{l}
\displaystyle\iint_{{\mathcal{R}}_R} \big(\tau(x)-\tau (y)\big)^2\,
v^2(y)\,u_2(x)\,u_2(y)\,K(x-y)\,dx\,dy\vspace{3pt}\\
\qquad\le\displaystyle C\,R^{-2}\iint_{{\mathcal{R}}_R}|x-y|^2 \,K(x-y)\,dx\,dy\vspace{3pt}\\
\qquad\displaystyle\le\,2\,C\,R^{-2}\int_{B_{2R}}\left[\int_{B_{R_0}}|z|^2\,K(z)\,dz\right]\,dx\, \le\,C,\end{array}$$
for some~$C>0$.
Therefore, recalling~\eqref{J22},
\begin{equation}\label{J23}
J_2^2 \le C \iint_{{\mathcal{R}}_R} \big(v(x)-v(y)\big)^2\,
\big(\tau(x)+\tau (y)\big)^2\,
\,u_2(x)\,u_2(y)\,K(x-y)\,dx\,dy.\end{equation}
Hence, since
$$ \big(\tau(x)+\tau (y)\big)^2=
\tau^2(x)+\tau^2 (y)+2\tau(x)\,\tau(y)\le
2\tau^2(x)+2\tau^2 (y),$$
we can use the symmetric role played by~$x$ and~$y$ in~\eqref{J23}
and obtain that
$$ J_2^2 \le C \iint_{{\mathcal{R}}_R} \big(v(x)-v(y)\big)^2\,
\tau^2(x)\,
\,u_2(x)\,u_2(y)\,K(x-y)\,dx\,dy,$$
up to renaming~$C>0$.
So, we insert this information into~\eqref{J24}
and we conclude that
\begin{equation}\label{J25}
\begin{split}
& \left[ \iint_{\R^2\times\R^2} \big(v(x)-v(y)\big)^2\,\tau^2(x) \,u_2(x)\,
u_2(y)\,K(x-y)\,dx\,dy \right]^2=J_1^2\\ &\quad\le J_2^2\le
C \iint_{{\mathcal{R}}_R} \big(v(x)-v(y)\big)^2\,
\tau^2(x)\,
\,u_2(x)\,u_2(y)\,K(x-y)\,dx\,dy,\end{split}\end{equation}
for some~$C>0$.

Since~${\mathcal{R}}_R\subseteq\R^2\times\R^2$ and $u_2$ and $K$ are nonnegative,
we can simplify the estimate in~\eqref{J25} by writing
$$ \iint_{\R^2\times\R^2} \big(v(x)-v(y)\big)^2\,\tau^2(x) \,u_2(x)\,
u_2(y)\,K(x-y)\,dx\,dy \le C.$$
In particular, since~$\tau=1$ in~$B_R$,
$$ \iint_{B_R \times B_R} \big(v(x)-v(y)\big)^2\,u_2(x)\,
u_2(y)\,K(x-y)\,dx\,dy \le C.$$
Since~$C$ is independent of~$R$, we can send~$R\to+\infty$
in this estimate and obtain that the map
$$ \R^2\times\R^2\ni (x,y)\mapsto
\big(v(x)-v(y)\big)^2\,u_2(x)\,
u_2(y)\,K(x-y)$$
belongs to~$L^1(\R^2\times\R^2)$.

Using this and the fact that~${\mathcal{R}}_R$ approaches the
empty set as~$R\to+\infty$, we conclude from Lebesgue's dominated convergence theorem that
$$ \lim_{R\to+\infty}
\iint_{{\mathcal{R}}_R} \big(v(x)-v(y)\big)^2 \,u_2(x)\,
u_2(y)\,K(x-y)\,dx\,dy =0.$$
Therefore, going back to~\eqref{J25} and recalling the properties of $\tau=\tau_R$,
\begin{eqnarray*}
&&
\left[ \iint_{\R^2\times\R^2} \big(v(x)-v(y)\big)^2 \,u_2(x)\,
u_2(y)\,K(x-y)\,dx\,dy \right]^2\\
&=& \lim_{R\to+\infty}
\left[ \iint_{\R^2\times\R^2} \big(v(x)-v(y)\big)^2\,\tau^2(x) \,u_2(x)\,
u_2(y)\,K(x-y)\,dx\,dy \right]^2\\ &\le&
\lim_{R\to+\infty}
C \iint_{{\mathcal{R}}_R} \big(v(x)-v(y)\big)^2\,
\tau^2(x)\,
\,u_2(x)\,u_2(y)\,K(x-y)\,dx\,dy.
\\ &=&0.\end{eqnarray*}
This and~\eqref{mono} imply that~$\big(v(x)-v(y)\big)^2 \,K(x-y)=0$
for a.e.~$(x,y)\in\R^2\times\R^2$.
Hence, recalling assumption~(H1), we have that~$v(x)=v(y)$
for any~$x\in\R^2$ and any~$y\in B_{r_0}(x)$.
As a consequence, the set~$\{y\in\R^2 {\mbox{ s.t. }} v(y)=v(0)\}$
is open and closed in~$\R^2$, and so, by connectedness,
we obtain that~$v$ is constant.\end{proof}

By combining Lemmata~\ref{L-A-1} and~\ref{Lio},
we can finish the proof of Theorems~\ref{DG-2} and \ref{DG-3}:

\begin{proof}[Completion of the proof of Theorems~$\ref{DG-2}$ and $\ref{DG-3}$]
Using first Lemma~\ref{L-A-1} and then Lemma~\ref{Lio},
we obtain that~$v$ is constant, where~$v$ is as in~\eqref{DF:v}.
Let us say that~$v(x)=a$ for some~$a\in\R$.
So we define~$\omega:=\frac{(a,1)}{\sqrt{a^2+1}}$ and we observe that
$$ \nabla u(x) = u_2(x)\,(v(x),1)=u_2(x)\,{\sqrt{a^2+1}}\;\omega.$$
Thus, if~$\omega\cdot y=0$
then
$$ u(x+y)-u(x)=\int_0^1 \nabla u(x+ty)\cdot y\,dt
=\int_0^1 u_2(x+ty)\,{\sqrt{a^2+1}}\;\omega\cdot y\,dt=0.$$
Therefore, if we
set~$u_\star(t):=u(t\omega)$
for any~$t\in\R$, and we write any~$x\in\R^2$ as
$$ x=\left({\omega}\cdot x\right) \omega+y_x$$
with~$\omega\cdot y_x=0$, we conclude that
$$ u(x)=u\left(
\left({\omega}\cdot x\right) \omega +y_x\right)
= u\left(
\left({\omega}\cdot x\right) \omega \right)
=u_\star\left({\omega}\cdot x\right).$$
This completes the proof of Theorem~\ref{DG-3}.\end{proof}

It is an interesting open problem to investigate
if symmetry results in the spirit of Theorems~\ref{DG-2} and~\ref{DG-3}
hold true in higher dimension.



\section{Stable solutions and extension of the main results}\label{EXT}

We discuss here the extension of Theorems \ref{DG-2} and \ref{DG-3} to the more general context of bounded \emph{stable} solutions $u$ of~\eqref{EQ} in the whole space $\R^n$ with $n\ge 2$.
In the case of second order equations, there are two equivalent definitions of stability: a variational one and a non-variational one.
In case of nonlocal operators \eqref{OP}, these two different definitions read as follows.
\begin{itemize}
\item[(S1)] The following inequality holds
\[\frac12\int_{\R^n}\int_{\R^n}\bigl( \xi(x)-\xi(x+y) \bigr)^2 K(y)\,dy\,dx \geq \int_{\R^n} f'(u)\xi^2\]
for every $\xi \in C^\infty_c(\R^n)$.
That is, the second variation of the energy functional associated to \eqref{EQ} is nonnegative under perturbations with compact support in $\R^n$.
\item[(S2)] There exists a positive continuous solution $\varphi>0$ to the linearized equation
\begin{equation}\label{linear}
\mathcal L\varphi=f'(u)\varphi\quad\textrm{in}\ \R^n.
\end{equation}
\end{itemize}

For completeness, we observe that a more general
version of Theorems~\ref{DG-2} and~\ref{DG-3} holds true,
namely if we replace assumption~\eqref{mono}
with the following non-variational stability condition~(S2).

\begin{thm}\label{DG-2-STAB}
Let $n=2$ and $\mathcal L$ be an operator of the form \eqref{OP}, with $K$ satisfying either \eqref{ASS-kernel}, or~\eqref{ASS-kernel-frac}, or {\rm{(H1)}}-{\rm{(H2)}}.
Let~$u$ be a solution of~\eqref{EQ}, with~$u\in C^1(\R^2)$ and~$f\in C^{1,\alpha}(\R)$, and with $u\in C^3(\R^2)$ in case {\rm{(H1)}}-{\rm{(H2)}}.
Assume that $u$ is stable, in the sense of~{\rm{(S2)}}.
Then, $u$ is necessarily one-dimensional.
\end{thm}

Notice that, in this setting, Theorems~\ref{DG-2} and \ref{DG-3}
are a particular case of Theorem~\ref{DG-2-STAB},
choosing~$\varphi:=u_2=\partial_{x_2}u$ and recalling~\eqref{5bis}.

The proof of Theorem~\ref{DG-2-STAB} is exactly the one
of Theorem~\ref{DG-3}, with only a technical difference:
instead of~\eqref{DF:v}, one has to define, for~$i\in\{1,2\}$,
$$ v(x):=\frac{u_i(x)}{\varphi(x)}.$$
Then the proof of Theorem~\ref{DG-3} goes through
(replacing~$u_2$ with~$\varphi$ when necessary)
and implies that~$v$ is constant, i.e.~$u_i =a_i\varphi$,
for some~$a_i\in\R$. This gives that~$\nabla u(x)=\varphi(x)\,(a_1,a_2)$,
which in turn implies the one-dimensional symmetry of~$u$.
\medskip

Given the result in Theorem~\ref{DG-2-STAB}, we discuss next the equivalence between the two definitions of stability (S1) and (S2). We will always assume that the kernel $K$ satisfies assumption~{\rm{(H1)}}.

\begin{prop}
Let $n\ge 1$ and $\mathcal L$ be any operator of the form \eqref{OP}.
Let~$u$ be a bounded solution of~\eqref{EQ} in the whole of~$\R^n$ with $f\in C^1(\R)$.
Assume that the kernel $K$ satisfies assumption {\rm{(H1)}}.
Then, {\rm{(S2)}} $\Longrightarrow$ {\rm{(S1)}}.
\end{prop}

\begin{proof}
Let $\xi\in C^\infty_0(\R^n)$.
Using $\xi^2/\varphi$ as a test function in the equation ${\mathcal{L}}\varphi=f'(u)\varphi$, we find
\[\int_{\R^n}f'(u)\xi^2=\int_{\R^n}\frac{\xi^2}{\varphi}\,
{\mathcal{L}}\varphi.\]
Next, we use \eqref{R1} (which holds in $\R^n$ as in $\R^2$) to see that at least at the formal level for any function $v$ and $w$ such that $\mathcal L w$ is well defined and $v$ belongs to $L^\infty(\R^n)$
\[\int_{\R^n}v\,{\mathcal{L}} w=\frac{B(v,w)}{2},\]
where
\[B(v,w):=\int_{\R^n}\int_{\R^n}\bigl(v(x)-v(y)\bigr)\bigl(w(x)-w(y)\bigr)K(x-y)\,dx\,dy.\]
We find (recall that $\varphi$ is such that $\mathcal L \varphi$ exists and $\xi$ is compactly supported)
\[2\int_{\R^n}f'(u)\xi^2=B\left(\varphi,\,\xi^2/\varphi\right).\]
Now, it is immediate to check that
\[\frac{\xi^2(x)}{\varphi(x)}-\frac{\xi^2(y)}{\varphi(y)}=\left(\xi^2(x)-\xi^2(y)\right)\frac{\varphi(x)+\varphi(y)}{2\varphi(x)\varphi(y)}
-\left(\varphi(x)-\varphi(y)\right)\frac{\xi^2(x)+\xi^2(y)}{2\varphi(x)\varphi(y)},\]
and this yields
\[\begin{split}
2\int_{\R^n}f'(u)\xi^2=&
\int_{\R^n}\int_{\R^n}\bigl(\varphi(x)-\varphi(y)\bigr)\left(\xi^2(x)-\xi^2(y)\right)\frac{\varphi(x)+\varphi(y)}{2\varphi(x)\varphi(y)}\,K(x-y)dx\,dy\\
&-\int_{\R^n}\int_{\R^n}\bigl(\varphi(x)-\varphi(y)\bigr)^2\ \frac{\xi^2(x)+\xi^2(y)}{2\varphi(x)\varphi(y)}K(x-y)dx\,dy.
\end{split}\]
Let us now show that
\begin{equation}\label{volem}
\begin{split}
&\Theta(x,y):=\bigl(\varphi(x)-\varphi(y)\bigr)\left(\xi^2(x)-\xi^2(y)\right)\frac{\varphi(x)+\varphi(y)}{2\varphi(x)\varphi(y)}\\
&\qquad\qquad\qquad\qquad\qquad-\bigl(\varphi(x)-\varphi(y)\bigr)^2\ \frac{\xi^2(x)+\xi^2(y)}{2\varphi(x)\varphi(y)}\leq \bigl(\xi(x)-\xi(y)\bigr)^2.
\end{split}\end{equation}
Once this is proved, then we will have
\[2\int_{\R^n} f'(u)\xi^2\leq \int_{\R^n}\int_{\R^n}\bigl( \xi(x)-\xi(y) \bigr)^2 K(x-y)dx\,dy,\]
and thus the result will be proved.

To establish \eqref{volem}, it is convenient to write $\Theta$ as
\[\begin{split}
&\Theta(x,y)=2\bigl(\varphi(x)-\varphi(y)\bigr)\bigl(\xi(x)-\xi(y)\bigr)\frac{\xi(x)+\xi(y)}{\varphi(x)+\varphi(y)}\cdot
\frac{\bigl(\varphi(x)+\varphi(y)\bigr)^2}{4\varphi(x)\varphi(y)}\\
&\qquad\qquad-\bigl(\varphi(x)-\varphi(y)\bigr)^2\cdot\left(\frac{\xi(x)+\xi(y)}{\varphi(x)+\varphi(y)}\right)^2 \frac{2\xi^2(x)+2\xi^2(y)}{\bigl(\xi(x)+\xi(y)\bigr)^2}\cdot\frac{\bigl(\varphi(x)+\varphi(y)\bigr)^2}{4\varphi(x)\varphi(y)}.
\end{split}\]
Now, using the inequality
\[\begin{split}
&2\bigl(\varphi(x)-\varphi(y)\bigr)\bigl(\xi(x)-\xi(y)\bigr)\frac{\xi(x)+\xi(y)}{\varphi(x)+\varphi(y)}\leq\\
&\qquad\qquad\qquad\qquad\qquad \leq \bigl(\xi(x)-\xi(y)\bigr)^2+\bigl(\varphi(x)-\varphi(y)\bigr)^2\cdot\left(\frac{\xi(x)+\xi(y)}{\varphi(x)+\varphi(y)}\right)^2,
\end{split}\]
we find
\[\begin{split}
&\Theta(x,y) \,\leq\, \displaystyle\bigl(\xi(x)-\xi(y)\bigr)^2\frac{\bigl(\varphi(x)+\varphi(y)\bigr)^2}{4\varphi(x)\varphi(y)}\,+\\
& \displaystyle+\bigl(\varphi(x)-\varphi(y)\bigr)^2\cdot\left(\frac{\xi(x)+\xi(y)}{\varphi(x)+\varphi(y)}\right)^2\cdot
\frac{\bigl(\varphi(x)+\varphi(y)\bigr)^2}{4\varphi(x)\varphi(y)}\cdot\left\{1-\frac{2\xi^2(x)+2\xi^2(y)}{\bigl(\xi(x)+\xi(y)\bigr)^2}\right\}.
\end{split}\]
But since
\[1-\frac{2\xi^2(x)+2\xi^2(y)}{\bigl(\xi(x)+\xi(y)\bigr)^2}=-\,\frac{\bigl(\xi(x)-\xi(y)\bigr)^2}{\bigl(\xi(x)+\xi(y)\bigr)^2},\]
we obtain
\[\begin{split}
\Theta(x,y)&\leq
\bigl(\xi(x)-\xi(y)\bigr)^2\frac{\bigl(\varphi(x)+\varphi(y)\bigr)^2}{4\varphi(x)\varphi(y)} - \bigl(\varphi(x)-\varphi(y)\bigr)^2\cdot \frac{\bigl(\xi(x)-\xi(y)\bigr)^2}{4\varphi(x)\varphi(y)}\\
&=\frac{\bigl(\xi(x)-\xi(y)\bigr)^2}{4\varphi(x)\varphi(y)}\left\{\bigl(\varphi(x)+\varphi(y)\bigr)^2-\bigl(\varphi(x)-\varphi(y)\bigr)^2\right\}\\
&= \bigl(\xi(x)-\xi(y)\bigr)^2.\end{split}\]
Hence \eqref{volem} is proved, and the result follows.
\end{proof}

Notice that the previous proposition holds for \emph{any} operator of the form \eqref{OP}, with no additional assumptions on $K$.
However, we do not know if the two stability conditions (S1) and (S2) are equivalent for all operators $\mathcal L$.
Indeed, in order to show the other implication (S1) $\Longrightarrow$ (S2), we need some additional assumptions.
Namely, we need:
\begin{equation}\label{H3.1}
\begin{split}
& {\mbox{if $w\in L^{\infty}(\R^n)$ is any weak solution to ${\mathcal{L}}w=g$ in $B_1$,
with $g\in L^\infty(B_1)$, then}}\\
& \qquad \|w\|_{C^\alpha(B_{1/2})}\leq
C\bigl(\|g\|_{L^\infty(B_1)}+
\|w\|_{L^\infty(\R^n)}\bigr)\\
& \hbox{for some constants $\alpha\in(0,1]$ and $C>0$ independent of $w$ and $g$}.\end{split}
\end{equation}
and
\begin{equation}\label{H3.2}
\begin{split}
& {\mbox{the space $H_K(\R^n)$, defined
as the closure of $C^\infty_0(\R^n)$ under the norm}}
\\ &
\|w\|_{H_K(\R^n)}^2:=\frac12\int_{\R^n}\int_{\R^n}\bigl(w(x)-w(y)
\bigr)^2 K(x-y)\,dx\,dy\\
&{\mbox{is compactly embedded in $L^2_{\rm loc}(\R^n)$.}}\end{split}
\end{equation}
\begin{rema}
These two assumptions \eqref{H3.1}-\eqref{H3.2} are satisfied for all kernels satisfying~\eqref{ASS-kernel-frac}.
Indeed, the $C^\alpha$ estimate \eqref{H3.1} can be found in \cite[Section 14]{CS1}, while the compact embedding \eqref{H3.2} follows easily in two steps: fix $p\in \R^n$ and use \eqref{ASS-kernel-frac} to have compactness in $L^2(B_{r_0/2}(p))$; then use a standard 
covering argument to have the compact embedding in $B_R$ $($for any $R>0$$)$.
See, for instance~\cite{MAZ} and~\cite[Theorem 7.1]{guide}
for further details on the compact embeddings.
\end{rema}

Using \eqref{H3.1}-\eqref{H3.2}, we have the following.

\begin{prop}\label{proS1S2}
Let $n\ge 1$ and $\mathcal L$ be any operator of the form~\eqref{OP} with kernel $K$ satisfying \eqref{ASS-kernel-frac}.
Let~$u$ be any bounded solution of~\eqref{EQ} in the whole of $\R^n$, with~$f\in C^{1,\alpha}(\R)$.
Then, {\rm{(S1)}} $\Longrightarrow$ {\rm{(S2)}}
\end{prop}

\begin{proof}
Let $R>0$ and consider the quadratic form
$$
\mathcal Q_R(\xi)=\frac12\int_{\R^n} \int_{\R^n}\bigl( \xi(x)-\xi(y)\bigr)^2 K(x-y)\,dx\,dy-\int_{B_R}f'(u)\xi^2\,dx,
$$
for $\xi\in C^{\infty}_0(\R^n)$.
Let~$H_K(\R^n)$ be as in~\eqref{H3.2}
and $\lambda_R$ be the infimum of $\mathcal Q_R$ among the class $\mathcal S_R$ defined by
$$
\mathcal S_R:= \left \{ \xi \in H_K(\R^n)
{\mbox{ s.t. }}
\xi= 0 \mbox{ in } \R^n \setminus B_R
{\mbox{ and }}
\int_{B_R} \xi^2 =1 \right \}.
$$
Since the functional $\mathcal Q_R$ is bounded from below in~${\mathcal{S}}_R$ (recall that~$f'(u)$ is bounded) and thanks to the compactness assumption in \eqref{H3.2}, we see that
its infimum $\lambda_R$ is attained for a function $\phi_R \in \mathcal S_R$.
Moreover, by assumption~(S1), we have
\begin{equation}\label{XC}
\lambda_R\geq0.\end{equation}
Also, we can assume that~$\phi_R \geq 0$,
since if $\phi$ is minimizer then $|\phi|$ is also a minimizer.
Thus, the function $\phi_R\geq0$ is a solution, not
identically zero, of the problem
\[\left \{
\begin{array}{rcll}
{\mathcal{L}} \phi_R & = & f'(u)\phi_R+\lambda_R \phi_R & \mbox{in}\,\,\,B_R,\vspace{3pt}\\
\phi_R & = & 0 & \mbox{in}\,\,\,\R^n \setminus B_R.
\end{array} \right.\]
It follows from the strong maximum principle for integro-differential operators 
(remember that $K$ satisfies~\eqref{ASS-kernel-frac}) that $\phi_R$ is continuous 
in $\R^n$ and $\phi_R>0$ in $B_R$. 
On the other hand, for any $0<R<R'$ we have
\[\int_{B_{R'}}\phi_{R}\;
{\mathcal{L}}\phi_{R'}=\int_{B_{R'}}\phi_{R'}\;
{\mathcal{L}}\phi_{R}<\int_{B_{R}} \phi_{R'}
\;{\mathcal{L}}\phi_{R}.\]
The equality above is a consequence of~\eqref{R1} (in $\R^n$),
while the inequality follows from the fact that $\phi_{R}=0$ in $B_{R'}\setminus B_{R}$, and thus ${\mathcal{L}}\phi_{R}<0$ in that annulus.
Hence, using the equations for $\phi_{R}$ and $\phi_{R'}$ we deduce that
\[\lambda_{R'}\int_{B_{R}}\phi_{R}\phi_{R'}<\lambda_{R}\int_{B_{R}}\phi_{R}\phi_{R'}.\]
Therefore, $\lambda_{R'}<\lambda_{R}$ for any $R'>R>0$. {F}rom this and~\eqref{XC}, it follows that
$\lambda_R>0$ for all $R>0$.

Now consider the problem
\begin{equation}\label{varphiR}\left \{
\begin{array}{rcll}
{\mathcal{L}}\varphi_R & = & f'(u)\varphi_R & \mbox{in}\,\,\,B_R,\vspace{3pt}\\
\varphi_R & = & c_R & \mbox{in}\,\,\,\R^n \setminus B_R,
\end{array} \right.
\end{equation}
for any fixed~$c_R>0$.
The solution to this problem can be found by writing $\psi_R=\varphi_R-c_R$, which solves
\[\left \{
\begin{array}{rcll}
{\mathcal{L}} \psi_R & = & f'(u)\psi_R+c_Rf'(u) & \mbox{in}\,\,\,B_R,\vspace{3pt}\\
\psi_R & = & 0 & \mbox{in}\,\,\,\R^n \setminus B_R.
\end{array} \right.\]
It is immediate to check that the energy functional associated to this last problem is bounded from below and coercive, thanks to the inequality $\lambda_R>0$. Therefore, $\psi_R$ and $\varphi_R$ exist.

Next we claim that $\varphi_R>0$ in $B_R$.
To show this, we use $\varphi_R^-$ as a test function for the equation for $\varphi_R$.
We find
\[\begin{split}
&\frac12\int_{\R^n} \int_{\R^n}\bigl(\varphi_R(x)-\varphi_R(y)\bigr)\bigl(\varphi_R^-(x)-\varphi_R^-(y)\bigr) K(x-y)\,dx\,dy\\
=\;&\int_{B_R}f'(u)\varphi_R\varphi_R^-\\
=\;&-\int_{B_R}f'(u)|\varphi_R^-|^2.
\end{split}\]
Now, since
\[\bigl(\varphi_R(x)-\varphi_R(y)\bigr)\bigl(\varphi_R^-(x)-\varphi_R^-(y)\bigr)\leq -\bigl(\varphi_R^-(x)-\varphi_R^-(y)\bigr)^2,\]
this yields
\[\mathcal Q_R(\varphi_R^-)=\frac12\int_{\R^n} \int_{\R^n}\bigl(\varphi_R^-(x)-\varphi_R^-(y)\bigr)^2 K(x-y)\,dx\,dy-\int_{B_R}f'(u)|\varphi_R^-|^2\,dx\leq 0.\]
Since $\lambda_R>0$, this means that $\varphi_R^-$ vanishes identically, and thus $\varphi_R\geq0$.
Since $K$ satisfies~\eqref{ASS-kernel-frac}, $\varphi_R$ is then continuous and 
positive in $\R^n$.
The above arguments also imply that the solution $\varphi_R$ of~\eqref{varphiR} is unique, whence $(1/c_R)\varphi_R$ is actually independent of $R>0$.
Therefore, one can choose the constant $c_R>0$ so that $\varphi_R(0)=1$.
Then, by the H\"older regularity in~\eqref{H3.1} and the Harnack inequality in~(H2),
we have that, for a sequence~$(R_k)_{k\in\mathbb{N}}\to+\infty$,
the functions $\varphi_{R_k}$ converge to a continuous function $\varphi>0$ in $\R^n$ and satisfying~\eqref{linear}.
\end{proof}

\end{document}